\documentclass[11pt]{amsart}

\newtheorem{theorem}{Theorem}[section]

\theoremstyle{rem}
\newtheorem{rem}[theorem]{Remark}
\newtheorem{lemma}[theorem]{Lemma}

\theoremstyle{definition}

\numberwithin{equation}{section}

\begin{document}

\author{E. Liflyand and U. Stadtm\"uller}
\dedicatory{Bar-Ilan University, Israel, and University of Ulm,
Germany}
\title{On a Hardy-Littlewood theorem}

\address{Department of Mathematics, Bar-Ilan University, 52900 Ramat-Gan, Israel}
\email{liflyand@math.biu.ac.il}
\address{Department of Mathematics, University of Ulm, 89069 Ulm, Germany}
\email{ulrich.stadtmueller@uni-ulm.de}

\begin{abstract}
A known Hardy-Littlewood theorem asserts that if both the function
and its conjugate are of bounded variation, then their Fourier
series are absolutely convergent. It is proved in the paper that the
same result holds true for functions on the whole axis and their
Fourier transforms with certain adjustments. The proof of the
original Hardy-Littlewood theorem is derived from the obtained
assertion. It turned out that the former is a partial case of the
latter when the function is supposed to be of compact support. A
similar result as the obtained one but for radial functions is
derived from the one-dimensional case.
\end{abstract}

\maketitle

\section{Introduction}

The following result is due to Hardy and Littlewood (see, e.g., \cite[Vol.I, Ch.VII,
(8.6)]{Zg}).

\begin{theorem}\label{fcbv} If a (periodic) function $f$ and its
conjugate $\widetilde f$ are both of bounded variation, their
Fourier series converge absolutely. \end{theorem}

In \cite{Zg} this result is one of the consequences of the general
theory of Hardy spaces, first of all $H^1$ in the unit disk. Since
we are going to generalize the Hardy-Littlewood theorem to functions on the
real axis, let us recall certain notions. The Fourier transform $\widehat f$
of a (complex-valued) function $g$ in $L^1(\mathbb R)$ is defined by

\begin{eqnarray*}\widehat g(t):=\int_{\mathbb R}g(x)e^{-itx}dx,\quad t\in\mathbb R,\end{eqnarray*}
while its Hilbert transform $\tilde g$ is defined by

\begin{align*}\mathcal{H}g(x)&:=\frac{1}{\pi}\,\mbox{\rm (P.V.)}
\int_{\mathbb R} g(x-u)\frac{du}{u}=\frac{1}{\pi}\,\mbox{\rm (P.V.)}
\int_{\mathbb R} \frac{g(u)}{x-u}\,du\\
&=\frac{1}{\pi}\lim\limits_{\delta\downarrow 0}\int_\delta^\infty
\{g(x-u)-g(x+u)\}\frac{du}{u},\quad  x\in\mathbb R.    \end{align*}
As is well known, for $g\in L^1(\mathbb R)$ this limit exists for
almost all $x$ in $\mathbb R,$ and the real Hardy space $H^1(\mathbb
R)$ is defined to be

\begin{eqnarray*}H^1(\mathbb  R):=\{g\in L^1(\mathbb  R):\mathcal{H}g\in L^1(\mathbb  R)\},\end{eqnarray*}
where $L^1(\mathbb  R)$ is the usual space of integrable functions with norm

\begin{eqnarray*}\|g\|_{L^1}:=\int_{\mathbb R}|g(x)|\,dx.\end{eqnarray*}
The Hardy space is endowed with the norm

\begin{eqnarray}\label{nh}\|g\|_{H^1}:=\|g\|_{L^1}+\|\mathcal{H} g\|_{L^1}.\end{eqnarray}
If $g\in H^1(\mathbb R)$, then

\begin{eqnarray}\label{vm}\int_{\mathbb R}  g(t)\,dt=0.\end{eqnarray}
It was apparently first mentioned in \cite{kober}.

Correspondingly, the absolute convergence of the Fourier series
should be replaced by the integrability of the Fourier transform.
Since a function $f$ of bounded variation may be not integrable, its
Hilbert transform, a usual substitute for the conjugate function,
may not exist. One has to use the modified Hilbert transform (see,
e.g., \cite{Garn})

$$\widetilde{f}(x)= \,\mbox{\rm (P.V.)}\,\frac{1}{\pi}\int_\mathbb{R}f(t)\biggl\{\frac{1}{x-t}+\frac{t}{1+t^2}\biggr\}\, dt$$
well adjusted for bounded functions. As a singular integral, it
behaves like the usual Hilbert transform; the additional term in the
integral makes it to be well defined near infinity.

Our work is much in the spirit of the book \cite{BN}, especially
Chapter 8. Roughly speaking, some classes are characterized there
for the values $r=1,2,...$ of a certain parameter $r$. Our
consideration formally corresponds to the case $r=0.$

The outline of the paper is as follows. In the next section we
formulate and prove the main result. As in the proof of the initial
result in \cite[Ch.VII, \S 8]{Zg} much is based on the Hardy
inequality (cf. (8.7) in the cited chapter and (\ref{Fein}) in the
present text). Then we derive the original Hardy-Littlewood theorem
from the proven result. In the last section we apply the obtained
theorem to deriving a similar multidimensional result for radial
functions.

\section{Main result}

\begin{theorem}\label{harlit} Let $f$ be a function of bounded variation
and vanish at infinity: $\lim\limits_{|t|\to\infty}f(t)=0.$ If its
conjugate $\widetilde f$ is also of bounded variation, then the
Fourier transforms of both functions are integrable on $\mathbb R.$
\end{theorem}

\begin{proof} The only property of a function of bounded variation we really need
is that its derivative exists almost everywhere and is integrable. So, such is $\frac{d}{dx}{\widetilde f}$.

More precisely, "almost everywhere" may be specified as the Lebesgue point. Recall that $x$ is
a Lebesgue point of an integrable function $g$ if $g(x)$ is finite and

\begin{eqnarray}\label{lp}\lim\limits_{t\to0}\frac{1}{t}\int_x^{x+t}|f(u)-f(x)|\,du=0.\end{eqnarray}

\begin{lemma}\label{dc} Under the assumptions of the theorem, we have at almost every $x$

\begin{eqnarray}\label{dc-cd}\frac{d}{dx}{\widetilde f}(x)
={\mathcal{H}}f'(x).\end{eqnarray} \end{lemma}

\begin{proof}[Proof of Lemma \ref{dc}] This lemma is a direct analog of the
various known results for the Hilbert transform of a function from
the spaces different from the space of functions of bounded
variation; see, e.g., \cite[3.3.1, Th.1]{Pand} or \cite[4.8]{King}.
Since the assumptions are different, we use different arguments
while interchanging limits. Let us start with the right-hand side of
(\ref{dc-cd}). Integrating by parts, we obtain

\begin{align}\label{ibp}{\mathcal{H}}f'(x)&=\lim\limits_{\delta\downarrow 0}\biggl(\int_{-\infty}^{x-\delta}
+\int_{x+\delta}^\infty\biggr)\frac{f'(t)}{x-t}\,dt\nonumber\\
&=\lim\limits_{\delta\downarrow0}\biggl[\frac{f(x-\delta)+f(x+\delta)}{\delta}-
\biggl(\int_{-\infty}^{x-\delta}+\int_{x+\delta}^\infty\biggr)\frac{f(t)}{(x-t)^2}\,dt\biggr].\end{align}
In order to integrate by parts, we need $f$ to be locally absolutely
continuous. It turns out that it is just the case under our
assumptions. Consider $F$ to be $f$ on a finite interval, say
$(-\pi,\pi]$ for simplicity, and zero otherwise. Its Hilbert
transform

$$\mathcal{H}F(x)=\frac{1}{\pi}\int_{-\pi}^\pi\frac{f(t)}{x-t}\,dt$$
differs from

$$\widetilde f(x)=\frac{1}{2\pi}\int_{-\pi}^\pi f(t)\cot\frac{x-t}{2}\,dt$$
(also understood in the principal value sense) only in kernels.
However (see, e.g., \cite[(9.0.4)]{BN}), the difference of these
kernels is quite good:

$$\frac12\cot\frac{t}{2}-\frac{1}{t}=\sum\limits_{k\ne0}\frac{t}{2k\pi(t-2k\pi)}.$$
In particular, the derivative of the right-hand side is integrable.
By this, if $\widetilde f$ is of bounded variation, then also
$\mathcal{H}F$ is. Observe, that for a function of compact support
there is no need in modifying the Hilbert kernel, and we can
consider the usual Hilbert transform. When a function and its
conjugate are both of bounded variation, the function is absolutely
continuous, see \cite[Ch.VII, (8.2)]{Zg}, exactly as required.

Further, since $f$ is continuous and vanishes at infinity we find

\begin{align*}&\frac{d}{dx}\biggl(\int_{-\infty}^{x-\delta}+\int_{x+\delta}^\infty\biggr)f(t)
\biggl(\frac{1}{x-t}+\frac{t}{1+t^2}\biggr)\,dt\\=&-\biggl(\int_{-\infty}^{x-\delta}+
\int_{x+\delta}^\infty\biggr)\frac{f(t)}{(x-t)^2}\,dt\\
&+f(x-\delta)\biggl[\frac1\delta+\frac{x-\delta}{1+(x-\delta)^2}\biggr]\\
&-f(x+\delta)\biggl[-\frac1\delta+\frac{x+\delta}{1+(x+\delta)^2}\biggr].\end{align*}
Combining the last two displays, we get

\begin{align*}{\mathcal{H}}f'(x)&=\lim\limits_{\delta\downarrow 0}\biggl\{
\frac{d}{dx}\biggl(\int_{-\infty}^{x-\delta}+\int_{x+\delta}^\infty\biggr)f(t)
\biggl(\frac{1}{x-t}+\frac{t}{1+t^2}\biggr)\,dt\\
&+f(x-\delta)\frac{x-\delta}{1+(x-\delta)^2}-f(x+\delta)
\frac{x+\delta}{1+(x+\delta)^2}\biggr\},\end{align*}
which gives

$${\mathcal{H}}f'(x)=\lim\limits_{\delta\downarrow 0}\,
\frac{d}{dx}\biggl(\int_{-\infty}^{x-\delta}+\int_{x+\delta}^\infty\biggr)f(t)
\biggl(\frac{1}{x-t}+\frac{t}{1+t^2}\biggr)\,dt,$$
since the rest, by continuity, tends to zero as $\delta\downarrow 0.$

What remains is to change the order of the limit and differentiation. The integrals
$\displaystyle{\int_{-\infty}^{x-1}+\int_{x+1}^\infty}$ converge uniformly, and
both are independent of $\delta$. Hence, it suffices to study the convergence of

$$\biggl(\int_{x-1}^{x-\delta}+\int_{x+\delta}^{x+1}\biggr)f(t)\biggl(\frac{1}{x-t}+\frac{t}{1+t^2}\biggr)\,dt.$$
Since

$$\biggl(\int_{x-1}^{x-\delta}+\int_{x+\delta}^{x+1}\biggr)\frac{dt}{x-t}=0,$$
and

$$\biggl(\int_{x-1}^{x-\delta}+\int_{x+\delta}^{x+1}\biggr)\frac{t}{1+t^2}\,dt\le2,$$
it remains to deal with

\begin{eqnarray}\label{last}\biggl(\int_{x-1}^{x-\delta}+\int_{x+\delta}^{x+1}\biggr)\frac{f(t)-f(x)}{x-t}\,dt.\end{eqnarray}
Let $x$ be a Lebesgue point of $f'$. We have

\begin{align*}\int_{x+\delta}^{x+1}\frac{f(t)-f(x)}{x-t}\,dt&=\int_{x+\delta}^{x+1}\frac{1}{x-t}
\int_x^t f'(u)\,du\,dt\\
&=-\int_\delta^1\frac{1}{t}\int_x^{x+t}f'(u)\,du\,dt.\end{align*}
Since $f'(x)$ is finite, we obtain

\begin{align*}\biggl|\int_{x+\delta}^{x+1}\frac{f(t)-f(x)}{x-t}\,dt\biggr|\le
\int_\delta^1\frac{1}{t}\int_x^{x+t}|f'(u)-f'(x)|\,du\,dt+|f'(x)|.\end{align*}
Since the inner integral on the right-hand side is uniformly bounded on $(0,1)$,
we get the uniform convergence of the integrals in (\ref{last}) (the other integral is
treated in exactly the same manner). This allows us
to apply $\lim\limits_{\delta\downarrow 0}$ and $\frac{d}{dx}$ in any order,
which leads to the required relation (\ref{dc-cd}) at any Lebesgue point of $f'$.
Since such is almost every point, the proof is complete.   \hfill\end{proof}

With this result in hand, the proof of the theorem continues as follows. Since the function $f$
is of bounded variation, its derivative $f'$ exists almost everywhere and is integrable.
It follows from the boundedness of variation of its conjugate and from Lemma \ref{dc} that
${\mathcal{H}}f'(x)$ exists at almost every $x$ and is also integrable. Therefore
$f'\in H^1(\mathbb R).$  We shall now make use of the well-known extension of Hardy's inequality
(see, e.g., \cite[(7.24)]{GR})

\begin{eqnarray}\label{Fein}\int_{\mathbb R}\frac{|\widehat{f'}(x)|}{|x|}\,dx
\le\|f'\|_{H^1(\mathbb R)}.\end{eqnarray}
Observe that the assumptions of the theorem imply the cancelation property (\ref{vm}) for $f'$.
Integrating by parts, which is possible since $f$ is locally absolutely continuous, we obtain

$$\widehat{f'}(x)=\int_{\mathbb R}f'(t)e^{-itx}dt=ix\int_{\mathbb R}f(t)e^{-itx}dt.$$
Hence, the left-hand side of (\ref{Fein}) is exactly the $L^1$ norm of the Fourier transform of $f$.
Further, we have $i\mbox{\rm sign}\,x\widehat{f'}(x)=\widehat{{\mathcal{H}}f'}(x)$, which, by
Lemma \ref{dc}, is $\displaystyle{\widehat{\frac{d}{dx}{\widetilde f}}}.$ Integrating by parts
as above, we conclude that in our situation the left-hand side of (\ref{Fein}) is exactly the $L^1$ norm
of the Fourier transform of $\widetilde{f}$. The proof is complete. \hfill\end{proof}

\section{The original Hardy-Littlewood theorem}

In this section we derive the proof of the original Hardy-Littlewood
theorem from Theorem \ref{harlit}. It turned out that the former is
a partial case of the latter when the function is supposed to be of
compact support.

Beginning the proof of Theorem \ref{fcbv}, we may consider $f$ to be
the $2\pi$-periodic extension of the function $F$ which coincides
with $f$ on $(-\pi,\pi]$ and is zero otherwise. Of course, this
function is of bounded variation as well. Using the argument after
(\ref{ibp}) in the opposite direction, we are now under the
assumptions of Theorem \ref{harlit}, and hence the Fourier transform
of $F$ is integrable. It follows from this (see \cite{W2}) that the
Fourier series of its periodic extension, that is, the Fourier
series of $f$, is absolutely convergent. Since the Fourier
coefficients of $\widetilde f$ are the same modulo as those of $f$,
the Fourier series of $\widetilde f$ also converges absolutely.

\section{Radial case}

Let us now make use of the obtained results in problems of integrability
of the multidimensional Fourier transform of a radial function
$f(x)=f_0(|x|).$ Let $\widehat f$ denote its usual Fourier transform
on $\mathbb R^n.$ The known Leray's formula (see Lemma 25.1' in
\cite{SKM}) says that when

\begin{eqnarray}\label{cosa} \int_0^\infty
|f_0(t)|\frac{t^{n-1}}{(1+t)^{\frac{n-1}{2}}}\,dt<\infty,\end{eqnarray}
the following relation holds

\begin{eqnarray}\label{samko}\widehat f(x)=2\pi^{\frac{n-1}{2}}
\int_0^\infty I(t)\cos|x|t\,dt,                 \end{eqnarray}
where the fractional integral $I$ is given by

\begin{eqnarray*} I(t)=\frac{2}{\Gamma\big(\frac{n-1}{2}\big)}\int_t^\infty
sf_0(s)(s^2-t^2)^{\frac{n-3}{2}}ds.                     \end{eqnarray*}
This result has proved to be very convenient for deriving statements
for the Fourier transform of a radial function from known
one-dimensional results; see, e.g., \cite{LT1}, \cite{LT2}.

\begin{theorem}\label{radalbo} Let $f_0$ satisfy (\ref{cosa}), while $I$
and its $n-1$ derivatives be locally absolutely continuous and
vanish at zero and at infinity. If $I^{(n-1)}$ satisfies the
assumptions of Theorem \ref{fcbv}, then the (multidimensional)
Fourier transform of $f$ is Lebesgue integrable over $\mathbb R^n.$
\end{theorem}

\begin{proof} First, we integrate by parts $n-1$ times in
(\ref{samko}). Integrated terms vanish. We thus arrive to the
formula

$$\widehat f(x)=\frac{2\pi^{(n-1)/2}(-1)^{n-1}}{|x|^{n-1}}
\int_0^\infty I^{(n-1)}(t)\cos(\frac{\pi(n-1)}{2}-|x|t)\,dt.  $$
Applying now Theorem \ref{fcbv} to the integral on the right-hand
side and integrating in the polar coordinates, we complete the
proof. \hfill\end{proof}

\begin{rem}
Observe that for $n=1$ understanding $I$ formally as $f_0$ reduces
Theorem \ref{radalbo} to the one-dimensional Theorem \ref{harlit}.
\end{rem}

\bibliographystyle{amsplain}

\begin{thebibliography}{99}


\bibitem{BN}
P.L. Butzer and R.J. Nessel, {\it Fourier analysis and
approximation}, {Volume 1: One-dimensional theory, Pure and Applied
Mathematics, Vol. 40},  {Academic Press},   {New York}, 1971.

\bibitem{Co} J. Cossar, {\it A theorem on Ces\`aro summability},
J. London Math. Soc. {\bf 16}(1941), 56--68.

\bibitem{GR}
J. Garcia-Cuerva and J.L. Rubio de Francia, {\it Weighted Norm
Inequalities and Related Topics}, North-Holland, 1985.

\bibitem{Garn} J.B. Garnett, {\it Bounded Analytic Functions},
Springer, N.Y., 2007.

\bibitem{King}
F.W. King, {\it Hilbert transforms, Vol.1}, Enc. Math/ Appl.,
Cambridge Univ. Press, Cambridge, 2009.

\bibitem{kober}
H. Kober, {\it A note on Hilbert's operator}, Bull. Amer. Math.
Soc., {\bf 48:1} (1942), 421--426.

\bibitem{LT1} E. Liflyand and S. Tikhonov, {\it Extended solution of Boas'
conjecture on Fourier transforms}, C. R. Acad. Sci. Paris, Ser. I,
{\bf 346} (2008), 1137--1142.

\bibitem{LT2} E. Liflyand and S. Tikhonov, {\it Two-sided weighted Fourier
inequalities}, Ann. Sc. Norm. Super. Pisa CI. Sci.(5). {\bf XI}
(2012), 341--362.


\bibitem{Pand} J.N. Pandey, {\it The Hilbert transform of Schwartz
distributions and applications},John Wiley \& Sons, New York, 1996.

\bibitem{SKM} S. G. Samko, A. A. Kilbas, O. I. Marichev,
{\it Fractional Integrals and Derivatives. Theory and
Applications}, Gordon\& Breach Sci. Publ., New York, 1992.




\bibitem{W2} N. Wiener, {\it The Fourier integral and certain of its
applications}, Dover Publ., Inc., New York, 1932.

\bibitem{Zg}  A. Zygmund, {\it Trigonometric series, Vol. I, II},
Cambridge Univ. Press, Cambridge, U.K.,  1966.

\end{thebibliography}

\end{document}